\documentclass{amsart}
\usepackage{amssymb}
\usepackage{amsfonts}
\usepackage{enumerate}
\usepackage{graphicx}

\newtheorem{theorem}{Theorem}
\theoremstyle{plain}

\newtheorem{corollary}{Corollary}

\newtheorem{definition}{Definition}
\newtheorem{example}{Example}

\newtheorem{lemma}{Lemma}

\numberwithin{equation}{section}

\begin{document}
\title{AG-groups as parallelogram spaces}
\author{M.~Shah}

\address{DEPARTMENT OF MATHEMATICS, QUAID-I-AZAM UNIVERSITY, ISLAMABAD,
PAKISTAN}
\email{mshahmaths@gmail.com}
\author{V.~Sorge}
\address{SCHOOL OF COMPUTER SCIENCE, UNIVERSITY OF BIRMINGHAM, UK.}
\email{V.Sorge@cs.bham.ac.uk}
\keywords{Medial quasigroups, AG-groups, parallelogram, parallelogram space}

\begin{abstract}
  It is known that an AG-group is paramedial and a paramedial is a parallelogram space.From which it follows that an AG-group is a parallelogram space. In this paper we give a direct proof of this fact and study it further. Our main result is that the parallelogram space of an AG-group is again an AG-group, which particularly implies that the parallelogram space for an Abelian group is also an Abelian group. We then generalise this result to medial quasigroups. Finally, we provide some quick methods of finding the other vertices of this parallelogram if at least one
  nontrivial vertex is known.
\end{abstract}

\subjclass[2010]{20N05}
\maketitle

\section{Introduction}

An AG-group $G$ is a groupoid having left identity and inverses as well as satisfying the left invertive law%
\begin{equation*}
(xy)z=(zy)x.
\end{equation*}%

A quasigroup is a groupoid ($Q,\cdot )$ such that for any $a,b$ in $G$ the
equations%
\begin{equation*}
a\cdot x=b,\qquad y\cdot a=b
\end{equation*}
have unique solutions for $x$ and $y$ lying in G. $x\cdot y$ is usually written
$xy$. The unique solutions $x$ and $y$ are sometimes denoted by left division
and right division as $x=a\backslash b$ and $y=b/a$, respectively.

A quasigroup $Q$ is called medial if the identity $ab\cdot cd=ac\cdot bd$
holds. If the additional identity $aa=a$ holds then it is called IM-quasigroup
(idempotent medial quasigroups)~\cite{VV3}. By Lemma~\ref{Pre-L1} , given below, it
is clear that every AG-group is medial. So all the geometric concepts that have
been introduced for a medial quasigroup in \cite{VV1,VK,NP} certainly hold for
AG-groups as well. It can be easily verified that an idempotent AG-group is an
Abelian group and therefore a non-associative AG-group cannot be idempotent and
hence cannot be a hexagonal quasigroup~\cite{VV4}, GS-quasigroup~\cite{VV2},
Steiner quasigroup~\cite{PR}, or quadratical quasigroup~\cite{VV5}. Thus
AG-groups are altogether a different subclass of medial quasigroups in contrast
to those subclasses of medial quasigroups, in which the concept of geometry has
been considered previously.

We first show that if $G$ is an AG-group then given any three points $a,b,c\in
G$ there exist a unique $d\in G$ such that $a,b,c,d$ form a parallelogram.
In~\cite{VV1} it has been shown that a medial quasigroup $Q$ is a parallelogram
space. We give a direct proof in Theorem~\ref{T2} of this fact for AG-groups
using the definition of parallelogram space from~\cite{OS}.

For various types of quasigroups, explicit formulae have been given to express
the fourth vertex of a parallelogram as a function of the other three (see
\cite{KV,VV2,VV7,VK}). We provide such a formula for AG-groups in
Theorem\ref{T1} together with an efficient way to compute it.

For some other classes of quasigroup one can provide methods to compute the
points of a parallelogram if at least two points are know. In our final result
we go beyond this by giving some methods of finding the remaining points of a
parallelogram if only one non-trivial point is known.
An AG-group $G$ is a groupoid having left identity and inverses as well as satisfying the left invertive law%
\begin{equation*}
(xy)z=(zy)x.
\end{equation*}%
AG-groups were first identified in~\cite{MK} as an important subclass of finite
Abel Grassmann Groupoids. They generalise the concept of Abelian groups in that
they not assume commutativity and associativity as both concepts are equivalent
for AG-group and if either one holds an AG-group is in fact an Abelian group
(see \cite{SA1} for further details). AG-groups are also sometimes called left
almost groups or LA-groups.

AG-groups have a versatile character such that in addition to their study in
parallel to groups, they can also be studied in the context of quasigroup and
loop theory~\cite{SA3}. In fact, AG-groups can be considered as a special type
of quasigroup.

A quasigroup is a groupoid ($Q,\cdot )$ such that for any $a,b$ in $G$ the
equations%
\begin{equation*}
a\cdot x=b,\qquad y\cdot a=b
\end{equation*}
have unique solutions for $x$ and $y$ lying in G. $x\cdot y$ is usually written
$xy$. The unique solutions $x$ and $y$ are sometimes denoted by left division
and right division as $x=a\backslash b$ and $y=b/a$, respectively.

A quasigroup $Q$ is called medial if the identity $ab\cdot cd=ac\cdot bd$
holds. If the additional identity $aa=a$ holds then it is called IM-quasigroup
(idempotent medial quasigroups)~\cite{VV3}. By Lemma ~\ref{Pre-L1} , given below, it
is clear that every AG-group is medial. So all the geometric concepts that have
been introduced for a medial quasigroup in \cite{VV1,VK,NP} certainly hold for
AG-groups as well. It can be easily verified that an idempotent AG-group is an
Abelian group and therefore a non-associative AG-group cannot be idempotent and
hence cannot be a hexagonal quasigroup~\cite{VV4}, GS-quasigroup~\cite{VV2},
Steiner quasigroup~\cite{PR}, or quadratical quasigroup~\cite{VV5}. Thus
AG-groups are altogether a different subclass of medial quasigroups in contrast
to those subclasses of medial quasigroups, in which the concept of geometry has
been considered previously.

We first show that if $G$ is an AG-group then given any three points $a,b,c\in
G$ there exist a unique $d\in G$ such that $a,b,c,d$ form a parallelogram.
In~\cite{VV1} it has been shown that a medial quasigroup $Q$ is a parallelogram
space. We give a direct proof in Theorem ~\ref{T2} of this fact for AG-groups
using the definition of parallelogram space from~\cite{OS}.

For various types of quasigroups, explicit formulae have been given to express
the fourth vertex of a parallelogram as a function of the other three (see
\cite{KV,VV2,VV7,VK}). We provide such a formula for AG-groups in
Theorem\ref{T1} together with an efficient way to compute it.

For some other classes of quasigroup one can provide methods to compute the
points of a parallelogram if at least two points are know. In our final result
we go beyond this by giving some methods of finding the remaining points of a
parallelogram if only one non-trivial point is known.

\section{Parallelograms}

Throughout the remainder of this article we will use the equations in Lemma ~\ref{Pre-L1} from~\cite{SA3} in our calculations.

\begin{lemma}
\label{Pre-L1} The following conditions holds in every AG-group $G$. Let
  $a,b,c,d,e\in G$, and let $e$ be the left identity in $G$.
\end{lemma}

\begin{enumerate}[(i)]
\item $(ab)(cd)=(ac)(bd)$,
\item $ab=cd\Rightarrow ba=dc$,
\item $a\cdot bc=b\cdot ac$,
\item $(ab)(cd)=(db)(ca)$,
\item $(ab)(cd)=(dc)(ba)$,
\item $ab=cd\Leftrightarrow d^{-1}b=ca^{-1}$,
\item If $e$ is a right identity in $G$ then it also becomes left identity in
  $G$, i.e., $ae=a\Rightarrow ea=a$,
\item $ab=e\Rightarrow ba=e$,
\item $(ab)^{-1}=a^{-1}b^{-1}$,
\item $a(b\cdot cd)=a(c\cdot bd)=b(a\cdot cd)=b(c\cdot ad)=c(a\cdot bd)=c(b\cdot ad)$,
\item $a(bc\cdot d)=c(ba\cdot d)$,
\item $(a\cdot bc)d=(a\cdot dc)b$,
\item $(ab\cdot c)d=a(bc\cdot d)$.
\end{enumerate}

\bigskip

Let $Q$ be a quasigroup. We shall say $a,b,c,d\in Q$ form a parallelogram, denoted by $Par(a,b,c,d)$, if there are points $p,q\in Q$ such that $pa=qb$ and
$pd=qc$.

\begin{theorem}
\label{T1} Let $G$ be an AG-group and $a,b,c,d\in G$. Then $Par(a,b,c,d)$ holds iff there are $x,y\in G$ such that $xb=a,by=c$ and $b\cdot xy=d$.
\end{theorem}
\begin{proof}
Let $x,y\in G$ be elements satisfying $xb=a,by=c$ and $b(xy)=d$. Let $e$
denotes the left identity in $G$. By taking $p=e$ and $q=x$, we see that $%
pa=qb$ and by Lemma ~\ref{Pre-L1} part(xiii) $pd=b(xy)=x(by)=qc$, i.e., $par(a,b,c,d)$
holds. Now suppose $par(a,b,c,d)$ holds and denote $x=a/b,y=b/c$ then $xb=a$
and $by=c$. According to[$3$, Corollary $5$], for any $p\in G$ there is a
unique $q\in G$ such that $pa=qb$ and $pd=qc$. Specially, for $p=e$ we see
that $a=qb\Rightarrow q=x$ and $d=qc=xc=x.by=b.xy$.
\end{proof}

\bigskip

In \cite{VV1}, it has been proved that if $Q$ is medial quasigroup then this quaternary
relation satisfies the following properties of parallelogram space.

\begin{enumerate}[(i)]
\item For any three points $a,b,c$ there is one and only one point $d$ such that
  $Par(a,b,c,d)$.
\item If $(e,f,g,h)$ is any cyclic permutation of $(a,b,c,d)$ or of $(d,c,b,a)$,
  then $Par(a,b,c,d)$ implies $Par(e,f,g,h)$.
\item From $Par(a,b,c,d)$ and $Par(c,d,e,f)$ it follows $Par(a,b,f,e)$.
\end{enumerate}

\begin{definition}
\label{def:paraspace}\cite{OS} A parallelogram space is a nonempty set $Q$ with quaternary relation $P\subseteq Q^{4}$ such that the following conditions are satisfied:
\end{definition}

\begin{enumerate}[(P1)]
\item If $Par(a,b,c,d)$ holds then $Par(a,c,b,d)$ holds for all $a,b,c,d\in Q$.
\item If $Par(a,b,c,d)$ holds then $Par(c,d,a,b)$ holds for all $a,b,c,d\in Q$.
\item If $Par(a,b,f,g)$ and $Par(f,g,c,d)$ hold then $Par(a,b,c,d)$ holds for
  all\newline $a,b,c,d,f,g\in Q$.
\item For any three points $a,b,c$ there is one and only one point $d$ such that
  $Par(a,b,c,d)$ holds.
\end{enumerate}

Again in \cite{VV1}, it has been shown that if $Q$ is medial quasigroup then the structure
(Q,Par) is a special case of Desargues systems in the terminology of D. Vakarelov \cite{Vak71} and
(Q,P) is a parallelogram space in the terminology of F. Ostermann and T. Schmit \cite{OS}, where
\[
P(a,b,c,d)\Leftrightarrow Par(a,b,d,c)  \tag{8}
\].
(Q,Par) is also called the parallelogram space for the sake of simplicity.

For an AG-group $G$, using definition \ref{def:paraspace} and $(8)$ we give a direct proof that (G,Par) is a parallelogram space.
But first we prove the following lemma.

\begin{lemma}
\label{newlemma}\begin{enumerate}[(i)]
\item If $Par(a,b,c,d)$ holds then $Par(c,d,a,b)$ holds for all $a,b,c,d\in G$.
\item For any three points $a,b,c$ there is one and only one point $d$ such that
$Par(a,b,c,d)$ holds.
\end{enumerate}
\end{lemma}

\begin{proof}
  \begin{enumerate}[(i)]
\item Let $Par(a,b,c,d)$ holds then there exist $p,q\in G$ such that the following holds:
  \begin{eqnarray*}
    pa = qb \mbox{ and } pd=qc
    &\Longrightarrow & qc=pd \mbox{ and } qb=pa \\
    &\Longrightarrow &Par(c,d,a,b)
\end{eqnarray*}
\item  \begin{enumerate}[(a)]
  \item Taking $d=cb^{-1}\cdot a$, we prove that $Par(a,b,c,d)$ holds. So let
    $p,q\in G$ such that $pa=qb$. Now
    \begin{equation*}
      pd=p(cb^{-1}\cdot a)=cb^{-1}\cdot pa=cb^{-1}\cdot qb=q(cb^{-1}\cdot b)=qc.
    \end{equation*}
    Thus $Par(a,c,b,d)$ holds.
  \item  For uniqueness, let $d_{1},d_{2}\in G$ be such that $pa=qb$ and $
    pd_{1}=qc,pd_{2}=qc$. From last two equations, by left cancellation we have
    $
    d_{1}=d_{2}$.
  \end{enumerate}
\end{enumerate}
\end{proof}

\begin{theorem}
\label{T2} An AG-group $G$ is a parallelogram space.
\end{theorem}

\begin{proof}
\begin{enumerate}[(P1)]
\item Let $P(a,b,c,d)$ holds therefore $Par(a,b,d,c)$ holds. Then by Theorem~\ref{T1} there exist
$x_{1},y_{1}\in G$ such that
\begin{equation*}
x_{1}b=a,by_{1}=d,b\cdot x_{1}y_{1}=c\Rightarrow x_{1}d=c.
\end{equation*}
Taking $x_{2}=ac^{-1},y_{2}=de\cdot c^{-1}$, we have $x_{2}c=a,cy_{2}=d$, and
$c\cdot x_{2}y_{2}=x_{2}\cdot cy_{2}=x_{2}d=ac^{-1}\cdot d=dc^{-1}\cdot a\linebreak{}=dc^{-1}\cdot x_{1}b=dx_{1}\cdot c^{-1}b=ce\cdot c^{-1}b=bc^{-1}\cdot c=b$.\\
Thus $Par(a,b,c,d)$ holds and hence $P(a,c,b,d)$ holds.
\item Let $P(a,b,c,d)$ holds. That is, $Par(a,b,d,c)$ holds. Then there exist $p,q\in G$ such that the following holds:
\begin{eqnarray*}
pa = qb \mbox{ and } pc=qd
&\Rightarrow & pc=qd \mbox{ and } pa=qb\\
\end{eqnarray*}
Thus $Par(c,d,b,a)$ holds and hence $P(c,d,a,b)$ holds.
\item Let $P(a,b,f,g)$ and $P(f,g,c,d)$ hold. That is, $Par(a,b,g,f)$ and $Par(f,g,d,c)$ hold. Then by Theorem~\ref{T1} there
exist $x_{1},y_{1},x_{2},y_{2}\in G$ such that
\begin{equation*}
x_{1}b=a,by_{1}=g,b\cdot x_{1}y_{1}=f\Rightarrow  x_{1}g=f \mbox{ and } x_{2}g=f,gy_{2}=d,g\cdot x_{2}y_{2}=c\Rightarrow x_{2}d=c.
\end{equation*}
Taking $x_{3}=x_{1}=x_{2},y_{3}=de\cdot b^{-1}$, we have $x_{3}b=a,by_{3}=d$, and
\begin{eqnarray*}
b\cdot x_{3}y_{3} &=&x_{3}b\{y_{3}=x_{3}d=x_{2}d=c.
\end{eqnarray*}
Thus again by Theorem~\ref{T1} we have proved that $Par(a,b,d,c)$ holds and hence $P(a,b,c,d)$ holds.
\item The proof is similar to Lemma~\ref{newlemma}(ii) with interchanging $p=q$.
\end{enumerate}
Hence $(G,Par)$ is a parallelogram space.
\end{proof}

\begin{corollary}
\label{C1} Let $G$ be an AG-group and let $Par(a,b,c,d)$ holds for some $p,q\in G$. Then
\begin{equation*}
ab^{-1}=p^{-1}q
\end{equation*}
\end{corollary}

\begin{proof}
First we prove that $Par(a,b,c,cq\cdot p^{-1})$ holds. Let $pa=qb$ for some $
p,q\in G$. Now
\begin{equation*}
p(cq\cdot p^{-1})=cq\cdot e=qc.
\end{equation*}
Thus $Par(a,b,c,cq\cdot p^{-1})$ holds. Now by Theorem~\ref{T2} (P4), we have
\begin{eqnarray*}
cb^{-1}\cdot a =cq\cdot p^{-1} &\Longrightarrow &ab^{-1}\cdot c=p^{-1}q\cdot c
\end{eqnarray*}
Hence by right cancellation, we have proved the claim.
\end{proof}

\bigskip

The above corollary provides us with a method of finding $p,q\in G$ if we
know that $Par(a,c,b,d)$ holds. Since G is a quasigroup so for $a,b\in G$ we
can find $p^{-1}q$. Then we can find $p^{-1}$ by fixing $q$ arbitrarily and
finally we have $p$ as the inverses of $p^{-1}$. The illustration is done in
the following example.

\begin{example}
\label{E1} \rm Consider the following AG-group of order $12$:
\end{example}

\begin{center}
\begin{tabular}{l|llllllllllll}
$\cdot $ & $0$ & $1$ & $2$ & $3$ & $4$ & $5$ & $6$ & $7$ & $8$ & $9$ & $10$
& $11$ \\ \hline
$0$ & $0$ & $1$ & $2$ & $3$ & $4$ & $5$ & $6$ & $7$ & $8$ & $9$ & $10$ & $11$
\\
$1$ & $1$ & $0$ & $3$ & $2$ & $6$ & $7$ & $4$ & $5$ & $10$ & $11$ & $8$ & $9$
\\
$2$ & $3$ & $2$ & $1$ & $0$ & $8$ & $11$ & $10$ & $9$ & $4$ & $7$ & $6$ & $5$
\\
$3$ & $2$ & $3$ & $0$ & $1$ & $10$ & $9$ & $8$ & $11$ & $6$ & $5$ & $4$ & $7$
\\
$4$ & $6$ & $4$ & $8$ & $10$ & $9$ & $0$ & $11$ & $1$ & $7$ & $2$ & $5$ & $3$
\\
$5$ & $7$ & $5$ & $11$ & $9$ & $0$ & $8$ & $1$ & $10$ & $3$ & $6$ & $2$ & $4$
\\
$6$ & $4$ & $6$ & $10$ & $8$ & $11$ & $1$ & $9$ & $0$ & $5$ & $3$ & $7$ & $2$
\\
$7$ & $5$ & $7$ & $9$ & $11$ & $1$ & $10$ & $0$ & $8$ & $2$ & $4$ & $3$ & $6$
\\
$8$ & $8$ & $10$ & $6$ & $4$ & $5$ & $2$ & $7$ & $3$ & $11$ & $0$ & $9$ & $1$
\\
$9$ & $9$ & $11$ & $5$ & $7$ & $3$ & $4$ & $2$ & $6$ & $0$ & $10$ & $1$ & $8$
\\
$10$ & $10$ & $8$ & $4$ & $6$ & $7$ & $3$ & $5$ & $2$ & $9$ & $1$ & $11$ & $
0 $ \\
$11$ & $11$ & $9$ & $7$ & $5$ & $2$ & $6$ & $3$ & $4$ & $1$ & $8$ & $0$ & $
10 $
\end{tabular}
\end{center}

Let us take $a=3,b=7$, then by Theorem ~\ref{T1} we have that $Par(3,7,2,6)$
holds. Now we want that for this parallelogram what $p,q$ actually are. So
by using Corollary ~\ref{C1}, we can do that in the following way.

We have $3\cdot 7^{-1}=p^{-1}q$ which implies that $8=p^{-1}q$. Now take $
q=10$ (say). So we have $8=p^{-1}\cdot 10$. this implies that $p^{-1}=1$.
This finally gives that $p=1^{-1}=1$. we can check them to be correct as $
2=1\cdot 3=10\cdot 7=2$ and $4=1\cdot 6=10\cdot 2=4$. Had we taken $q=4$, we
would have gotten $p=3$.

\bigskip

Theorem~\ref{T2} (P4) provides the fourth point or element $d$, for any
three points elements $a,b,c$ of $G$ to form a parallelogram. The following
theorems describes how to form a parallelogram if any two points are given.

\begin{theorem}
\label{T3} Let $G$ be an AG-group and $a,b,p\in G$. Then $Par(a,b,pb,pa)$
holds.
\end{theorem}
\begin{proof}
Let $c=pb,d=pa$. Take $q\in G$ such that $pa=qb$. Now $
qc=q(pb)=p(qb)=q(pa)=pd$. Hence $Par(a,b,pa,pb)$ holds.
\end{proof}

The following is a diagrammatic depiction of the result of Theorem~\ref{T3}:

\begin{center}
\includegraphics[height=3.5cm]{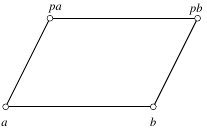}
\end{center}

Let us also illustrate the above theorem by an example.

\begin{example}
Take arbitrarily $a=2,b=7,p=9$ in Example~\ref{E1}. Then we can
find $q=1$ such that $Par(2,7,6,5)$ holds. Note that $p=9$ is already
understood.
\end{example}

\begin{theorem}
\label{T4} Let $G$ be an AG-group. Then $Par(a,b,a^{-1},(ab)^{-1}a)$ holds.
\end{theorem}

\begin{proof}
Let $p,q\in G$ such that $pa=qb$.
\begin{eqnarray*}
\text{Now }qa^{-1} &=&[(b^{-1}a)p]a^{-1} \\
&=&(a^{-1}p)(b^{-1}a)\text{by invertive law} \\
&=&(a^{-1}b^{-1})(pa)\text{ by Lemma~\ref{Pre-L1} Part(i)} \\
&=&p[(a^{-1}b^{-1})a]\ \text{by Lemma~\ref{Pre-L1} Part(iii)} \\
&=&p[(ab)^{-1}a]\ \text{by Lemma~\ref{Pre-L1} Part(ix)}
\end{eqnarray*}
Hence $Par(a,b,a^{-1},(ab)^{-1}a)$ holds.
\end{proof}

Again we will illustrate the above theorem by a diagrammatic depiction and then
by an example.\vspace{.5cm}

\begin{center}
  \includegraphics[height=4.5cm]{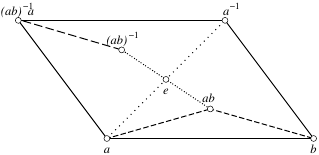}
\end{center}

\begin{example}
Take arbitrarily $a=3,b=8$ in Example~\ref{E1}. Then $Par(3,8,2,11)$
holds. Here $p=1$ and $q=7$.
\end{example}

\begin{theorem}
  \label{T5} Let $G$ be an AG-group and let $a,b\in G$ such that $e\neq a$. Then
  $Par(a,ab,(ae)a^{-1},b)$ holds
\begin{proof}
Let $p,q\in G$ such that $pa=q(ab)$. Now
\begin{eqnarray*}
pa &=&a(qb)\ \text{by Lemma~\ref{Pre-L1} Part(xiii)} \\
ap &=&(qb)a\ \text{by Lemma~\ref{Pre-L1} Part(ii)} \\
qb &=&(ap)a^{-1}\text{by right inverse property} \\
qb &=&(ap)(ea^{-1}) \\
qb &=&(a^{-1}e)(pa)\text{\ by Lemma~\ref{Pre-L1} Part(iv)} \\
qb &=&p((a^{-1}e)a)\text{\ by Lemma~\ref{Pre-L1} Part(iii)} \\
qb &=&p((ae)a^{-1}).
\end{eqnarray*}
\newline
Hence $Par(a,ab,(ae)a^{-1},b)$ holds.
\end{proof}
\end{theorem}

Let us consider the following example.

\begin{example}
Take arbitrarily $a=5,b=2$ in Example ~\ref{E1}. Then $Par(5,11,1,2)$
holds. Here $p=10$ and $q=4$.
\end{example}

The following corollary provides a fast method to compute the parallelogram
space for any one non-trivial element of an AG-group $G$.

\begin{corollary}
\label{C2} Let $G$ be an AG-group and let $a\in G$ such that $e\neq a$. Then
$Par(a,ae,(ae)a^{-1},e)$ holds.
\end{corollary}
Observe that since $e$ is the left unit element in the AG-group, the element
$ae$ is in generally different from $a$ and thus $Par(a,ae,(ae)a^{-1},e)$ is
non-trivial. We again illustrate the corollary with a diagram followed by a
concrete example.
\begin{center}
\includegraphics[height=3.5cm]{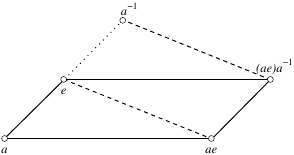}
\end{center}

\begin{example}
Take arbitrarily $a=6$ in Example~\ref{E1}. Then $Par(6,4,1,0)$ holds.
Here $p=3$ and $q=2$.
\end{example}

\begin{theorem}
  \label{T6} Let $G$ be an AG-group, if $Par(a_{1},b_{1},c_{1},d_{1})$ and
  $Par(a_{2},b_{2},c_{2},d_{2})$ then
  $Par(a_{1}a_{2},b_{1}b_{2},c_{1}c_{2},d_{1}d_{2})$ also holds.
\end{theorem}

\begin{proof}
  Since $Par(a_{1},b_{1},c_{1},d_{1})$ and $Par(a_{2},b_{2},c_{2},d_{2})$ hold,
  by Theorem ~\ref{T1} there exist $x_{1},y_{1},x_{2},y_{2}\in G$ such that
\begin{equation*}
x_{1}a_{1}=b_{1},b_{1}y_{1}=c_{1},b_{1}\cdot
x_{1}y_{1}=d_{1},x_{2}a_{2}=b_{2},b_{2}y_{2}=c_{2},b_{2}\cdot
x_{2}y_{2}=d_{2}
\end{equation*}
Taking $x_{3}=x_{1}x_{2},y_{3}=c_{2}b_{1}^{-1}\cdot c_{1}b_{2}^{-1}$, we
have $x_{3}\cdot a_{1}a_{2},b_{1}b_{2}\cdot y_{3}=c_{1}c_{2}$. Now
\begin{eqnarray*}
b_{1}b_{2}\cdot x_{3}y_{3}
&=&(b_{1}b_{2})\{(x_{1}x_{2})(c_{2}b_{1}^{-1}\cdot c_{1}b_{2}^{-1})\} \\
&=&(b_{1}b_{2})\{(x_{1}x_{2})(c_{2}c_{1}\cdot b_{1}^{-1}b_{2}^{-1})\} \\
&=&(x_{1}x_{2})\{(c_{2}c_{1})(b_{1}b_{2}\cdot b_{1}^{-1}b_{2}^{-1})\} \\
&=&(x_{1}x_{2})(c_{2}c_{1}\cdot e) \\
&=&x_{1}x_{2}\cdot c_{1}c_{2} \\
&=&x_{1}c_{1}\cdot x_{2}c_{2} \\
&=&(b_{1}\cdot x_{1}y_{1})\cdot (b_{2}\cdot x_{2}y_{2}) \\
&=&d_{1}d_{2}
\end{eqnarray*}
Hence by Theorem ~\ref{T1} we are done.
\end{proof}

\begin{theorem}
\label{T7} Let $G$ be an AG-group then the parallelogram space ($G,Par)$ is
again an AG-group.
\end{theorem}

\begin{proof}
Define a binary operation $@$ on $(G,Par)$ by
\begin{eqnarray*}
Par(a_{1},b_{1},c_{1},d_{1})@Par(a_{2},b_{2},c_{2},d_{2})
&=&Par(a_{1}a_{2},b_{1}b_{2},c_{1}c_{2},d_{1}d_{2}) \\
& & \quad \text{for all }a_{1},b_{1},c_{1},d_{1,}a_{2},b_{2},c_{2},d_{2} \in G
\end{eqnarray*}
By Theorem ~\ref{T6} ($G,Par)$ is closed under $@$.

\noindent Let
$x=Par(a_{1},b_{1},c_{1},d_{1}),y=Par(a_{2},b_{2},c_{2},d_{2}),z=Par(a_{3},b_{3},c_{3},d_{3})
{\in} (G,Par)$ then
\begin{eqnarray*}
(x@y)@z
&=&[Par(a_{1},b_{1},c_{1},d_{1})@Par(a_{2},b_{2},c_{2},d_{2})]@Par(a_{3},b_{3},c_{3},d_{3})
\\
&=&Par(a_{1}a_{2},b_{1}b_{2},c_{1}c_{2},d_{1}d_{2})@Par(a_{3},b_{3},c_{3},d_{3})
\\
&=&Par(a_{1}a_{2}\cdot a_{3},b_{1}b_{2}\cdot b_{3},c_{1}c_{2}\cdot
c_{3},d_{1}d_{2}\cdot d_{3}) \\
&=&Par(a_{3}a_{2}\cdot a_{1},b_{3}b_{2}\cdot b_{1},c_{3}c_{2}\cdot
c_{1},d_{3}d_{2}\cdot d_{1}) \\
&=&Par(a_{3}a_{2},b_{3}b_{2},c_{3}c_{2},d_{3}d_{2})@Par(a_{1},b_{1},c_{1},d_{1})
\\
&=&[Par(a_{3},b_{3},c_{3},d_{3})@Par(a_{2},b_{2},c_{2},d_{2})]@Par(a_{1},b_{1},c_{1},d_{1})
\\
&=&(z@y)@x
\end{eqnarray*}
Thus ($G,Par)$ is an AG-groupoid under $@$.

\noindent $Par(e,e,e,e)\in (G,Par)$ plays the role of left identity as for all $
Par(a,b,c,d)\in (G,Par)$, we have
\begin{equation*}
Par(e,e,e,e)@Par(a,b,c,d)=Par(ea,eb,ec,ed)=Par(a,b,c,d)
\end{equation*}
Every element $Par(a,b,c,d)$ in ($G,Par)$ has an inverse $Par(a^{-1},c,b,d)$
as
\begin{eqnarray*}
Par(a,b,c,d)@Par(a^{-1},b^{-1},c^{-1},d^{-1}) &=&Par(e,e,e,e)\text{ and} \\
Par(a^{-1},b^{-1},c^{-1},d^{-1})@Par(a,b,c,d) &=&Par(e,e,e,e)
\end{eqnarray*}
Hence ($G,Par,@)$ is an AG-group.
\end{proof}

\begin{corollary}
  Let $G$ be an abelian group, then ($G,Par,@)$ is an abelian group.
\end{corollary}

We can now generalize the above results to the following theorem.
\begin{theorem}
  Let $G$ be a medial quasigroup, then ($G,Par,@)$ is also a medial quasigroup.
\end{theorem}
\noindent The proof of this theorem is analogous to that of Theorem~\ref{T7}.


\end{document}